\documentclass[10pt]{amsart}

\usepackage{amssymb}

\newtheorem{thm}{Theorem}[section]
\newtheorem{prop}[thm]{Proposition}
\newtheorem{lem}[thm]{Lemma}

\theoremstyle{defi}
\newtheorem{defi}[thm]{Definition}

\theoremstyle{rem}
\newtheorem{rem}[thm]{Remark}

\numberwithin{equation}{section}

\begin{document}

\title{Small perturbation solutions for nonlocal elliptic equations}

\author{Hui Yu}
\address{Department of Mathematics, The University of Texas at Austin}
\email{hyu@math.utexas.edu}

\begin{abstract}
We present a small perturbation result for nonlocal elliptic equations, which says that for a class of nonlocal operators, the solutions are in $C^{\sigma+\alpha}$ for any $\alpha\in(0,1)$ as long as the solutions are small. This is a nonlocal generalization of a celebrated result of Savin in the case of second order equations.
\end{abstract}

\maketitle

\tableofcontents

\section{Introduction}
In this work, we present a nonlocal generalization of a celebrated result by Savin concerning small perturbation solutions for elliptic equations \cite{S}, which, in its simplest form, states the following:
\begin{thm} Suppose $F$ is a `nice' uniformly elliptic operator. For any $\alpha\in(0,1)$ there is a constant $\kappa>0$ such that if $u$ solves in the viscosity sense
$$F(D^2u)=0 \text{  in $B_1$},$$ then $$\|u\|_{C^{2,\alpha}(B_{1/2})}\le C(\alpha)\|u\|_{\mathcal{L}^{\infty}(B_1)}$$ whenever $\|u\|_{\mathcal{L}^{\infty}(B_1)}\le\kappa$.
\end{thm} 

Here a `nice' operator enjoys certain regularity properties near $0$. To be precise, $F$ is required to be $C^2$ with bounded Hessian in \cite{S}. This condition can be relaxed, see for example \cite{ASS}.

Compare with the classical result by Evans \cite{E} and Krylov \cite{K}, which gives a $C^{2,\alpha}$-estimate for concave operators, Savin's replaces a structural concavity condition on the operator by a regularity condition together with a smallness condition on the solution. Its significance lies in the fact that actual solutions often reduce to small perturbation solutions once one subtracts some standard objects, e.g. Taylor polynomials, supporting parabola, etc. As a result, there have been numerous applications of this result. To name a few,  Armstrong-Silvestre-Smart on partial regularity \cite{ASS},  Armstrong-Silvestre on unique continuation \cite{AS}, Collins on $C^{2,\alpha}$ estimates of equations of twisted type \cite{C}. 

We'd like to point out that for nonlocal equations, partial regularity and unique continuation remain open, and are actually the main motivation for this work. If one follows the strategy of Armstrong-Silvestre-Smart and Armstrong-Silvestre, the main ingredients needed are a small perturbation result, a unique continuation result for linear operators \cite{FF} \cite{Se} \cite{R} and a $W^{\sigma,\epsilon}$-estimate \cite{Yu}. 

For second order equations, the small perturbation result follows from a simple compactness argument. Suppose the result is false, then one finds a sequence of operators $F_k$ and a sequence of solutions $u_k$ such that $$F_k(D^2u_k)=0 \text{ in $B_1$}$$ and $$\|u_k\|_{\mathcal{L}^{\infty}(B_1)}=\kappa_k\to 0,$$ but $$\|u_k\|_{C^{2,\alpha}(B_{1/2})}\ge k\kappa_k.$$ In particular $\tilde{u}_k:=u_k/\kappa_k$ satisfies $$\frac{1}{\kappa_k}F_k(\kappa_kD^\sigma \tilde{u}_k)=0,$$ $$\|\tilde{u}_k\|_{\mathcal{L}^{\infty}(B_1)}=1$$ and $$\|\tilde{u}_k\|_{C^{2,\alpha}(B_{1/2})}\ge k.$$ 

Now with the regularity assumptions on the operators, $\frac{1}{\kappa_k}F_k(\kappa_k\cdot)$ converges to a constant coefficient linear operator, the `derivative' at $0$. Meanwhile, uniform H\"older estimate gives a locally uniform limit $u^*$ of $\{\tilde{u}_k\}$. Stability says $u^*$ solves the constant coefficient linear elliptic equation and hence enjoys extremely nice regularity properties, contradicting the large $C^{2.\alpha}$-norm of $\tilde{u}_k$. For details on uniform H\"older estimate as well as the stability for elliptic equations, one can consult \cite{CC}.

This gives the desired result for second order equations, which says that `nice' dependence on the Hessian matrix and smallness of the solution imply regularity of the solution. For a nonlocal version of this proof, however, one faces several difficulties. Firstly, one needs a fractional order replacement for the Hessian matrix. This should be some object that records  directional fractional order curvatures. Although there is no canonical way to define such objects at this stage, the operator $D^{\sigma}$, as studied in \cite{Yu2}, seems natural for this task, its definition  given in the next section.

Another major difficulty for nonlocal operators comes from the failure of the compactness argument.  For second order equations, a uniform $\mathcal{L}^\infty$ control over the solutions gives compactness of the sequence in the \textit{entire domain} via the H\"older estimate. For instance, solutions to equations in $B_1$ will converge locally uniformly in the entire $B_1$. For fractional order equations, however, the best one can hope for is convergence in the \textit{domain of the equation}, while there is no hope for convergence on complement of that domain if one only assumes $\mathcal{L}^\infty$ control. Solutions to equations in $B_1$ will converge in $B_1$, but these solutions are defined in the entire $\mathbb{R}^n$ and there is no convergence in $\mathbb{R}^n\backslash B_1$. The reason for this contrast is that the boundary data for second order equations live in lower dimensional sets while the boundary data for fractional order equations live in sets of full dimensions.

Fortunately for us, in a series of papers \cite{Ser1}\cite{Ser2} Serra showed one possible way to deal with this problem. The new ingredient is a blow-up type argument. By zooming in at certain points along the sequence, the scaled sequence solve equations in larger and larger domains that converge to the entire space. Furthermore, if one zooms in at the right rate, the scaled sequence enjoy certain growth condition. This gives a limit that is a global solution. For such solutions the contrast between local and nonlocal equations become less problematic since effectively there is no `boundary'. Then one uses a Liouville type theorem to characterize all possible global profiles, and closeness to these global profiles gives regularity for the solution to the original problem. 

We very much follow this type of argument with certain modifications to prove the following main result of this work:

\begin{thm}
Let  $F:S^n\to \mathbb{R}$ be a uniformly elliptic operator with ellipticity constants $\lambda$ and $\Lambda$. Assume it is continuously differentiable, and $DF$ has modulus of continuity $\omega$. $F(0)=0$.

Given $\alpha\in(0,1)$ and $\sigma\in(0,2)$ such that $\sigma+\alpha$ is not an integer, there exist constants $\kappa=\kappa(n, \lambda,\Lambda,\omega,\alpha,\sigma)>0$ and $C=C(n, \lambda,\Lambda,\omega,\alpha,\sigma)<\infty$ such that the viscosity solution to $$F(D^\sigma u)=0 \text{ in $B_1$}$$ satisfies the estimate $$\|u\|_{C^{\sigma+\alpha}(B_{1/2})}\le C(\|u\|_{\mathcal{L}^{\infty}(B_1}+\|u\|_{\mathcal{L}^1(\frac{1}{1+|y|^{n+\sigma}}dy)})$$ whenever $$\|u\|_{\mathcal{L}^{\infty}(B_1)}+\|u\|_{\mathcal{L}^1(\frac{1}{1+|y|^{n+\sigma}}dy)}\le\kappa.$$ \end{thm}   

Here $S^n$ is the space of $n\times n$ symmetric matrices. $\|\cdot\|_{C^{\sigma+\alpha}}$ is to be understood as $\|\cdot\|_{C^{\nu,\beta}}$ where $\nu$ is the integer part of $\sigma+\alpha$ and $\beta=\sigma+\alpha-\nu$. For definition of viscosity solutions to nonlocal operators, see \cite{CS1}.

We'd also like to point out that with the method in this paper, it is not too difficult to cover operators that involve lower order terms or more general $\sigma$-order Hessian matrices with weights as considered in \cite{Yu2}. It is also possible to prove uniform estimates as $\sigma\to 2$ and hence recover the result of Savin in the limit. However the proof is already rather involved, and we decide not to pursuit these interesting points.

This paper is organized as follows: In the next section we recall some definitions and preliminary results that will be useful for later sections. We also prove a Liouville type theorem. It is used in the third section to study blow-up limits of solutions, and to give an improvement of regularity. In the last section of this paper this improvement of regularity is combined with previously known regularity estimates to complete the proof for the main result.

\section{Preliminaries and the Liouville theorem}

We first define our replacement of the Hessian matrix. For each $(i,j)$ it records the $\sigma$-order curvature in the direction $(e_i,e_j)$. The reader should consult \cite{Yu2} for a more general version that involves weights.

\begin{defi}
For $u:\mathbb{R}^n\to\mathbb{R}$ with $$\int|\delta u(x,y)|\frac{1}{|y|^{n+\sigma}}dy<\infty,$$ its $\sigma$-order Hessian at $x$, $D^{\sigma}u(x)$, is a matrix with $(i,j)$-entry  $$\int\delta u(x,y)\frac{\langle e_i,y\rangle\langle e_j,y\rangle}{|y|^{n+\sigma+2}}dy.$$
\end{defi}

Here $\{e_i\}_{1\le i\le n}$ is the standard basis for $\mathbb{R}^n$. $\delta u(x,y)=u(x+y)+u(x-y)-2u(x)$ is the symmetric difference.

The following is a compactness result for sequence of operators.

\begin{prop}
Let $F_k:S^n\to\mathbb{R}$ be a sequence of uniformly elliptic operators with the same elliptic constants $\lambda$ and $\Lambda$. $F_k(0)=0$. Suppose $F_k$ are $C^1$ with $DF_k$ satisfying modulus of continuity $\omega$.

For any $\delta_k\to 0$, define $E_k:S^n\to\mathbb{R}$ by $$E_k(M)=\frac{1}{\delta_k}F_k(\delta_k M).$$

Then up to a subsequence $E_k$ converges locally uniformly to a constant coefficient linear operator.
\end{prop}

\begin{proof}
Since $\lambda\le DF_k(0)\le\Lambda$, up to a subsequence one has $DF_k(0)\to A\in S^n.$ We prove that up to a subsequence, $E_k$ converges to the linear operator $M\in S^n\mapsto\Sigma A_{ij}M_{ij}\in\mathbb{R}$.

\begin{align*}|E_k(M)-\Sigma A_{ij}M_{ij}|&=|\frac{1}{\delta_k}F_k(\delta_k M)-\Sigma A_{ij}M_{ij}|\\&=|\frac{1}{\delta_k}(F_k(\delta_k M)-F_k(0))-DF_k(0)M+DF_k(0)M-\Sigma A_{ij}M_{ij}|\\&\le|\frac{1}{\delta_k}(F_k(\delta_k M)-F_k(0))-DF_k(0)M|+|DF_k(0)-A||M|.
\end{align*} Note that $$F_k(\delta_k M)-F_k(0)=\int_0^1\frac{d}{dt}F_k(t\delta_kM)dt=\int_0^1DF_k(t\delta_kM)dt\cdot\delta_kM,$$ regularity of $DF_k$ leads to the following estimate
\begin{align*}|\frac{1}{\delta_k}(F_k(\delta_k M)-F_k(0))-DF_k(0)M|&=|\int_0^1DF_k(t\delta_kM)dt\cdot M-DF_k(0)M|\\&\le|\int_0^1DF_k(t\delta_kM)-DF_k(0)dt\cdot M|\\&\le\omega(\delta_k|M|)|M|.
\end{align*}
Consequently $$|E_k(M)-\Sigma A_{ij}M_{ij}|\le\omega(\delta_k|M|)|M|+|DF_k(0)-A||M|.$$
\end{proof} 

We will also need the following stability result for our operators. For some related stability result for nonlocal operators, see \cite{CS2}.

\begin{prop}
Let $F_k:S^n\to\mathbb{R}$ be a sequence of uniformly elliptic operators with the same ellipticity constants. $u_k\in\mathcal{L}^\infty(\mathbb{R}^n)\cap C(\bar{B_1})$ are viscosity solutions to $$F_k(D^\sigma u_k)=0 \text{ in $B_1$}.$$If $F_k\to F$ locally uniformly over $S^n$, and $u_k\to u$ locally uniformly over $\mathbb{R}^n$, then $$F(D^\sigma u)=0 \text{ in $B_1$}.$$
\end{prop}

\begin{proof}
Let $\phi$ be a smooth function \textit{strictly} touching $u$ from above at some $x_0\in B_1$. 

For some small $r>0$ define $$\tilde{\phi}(x)=\begin{cases}\phi(x), &x\in B_r(x_0)\\u(x), &x\not\in B_r(x_0)\end{cases},$$ and  $$\tilde{\phi}_k(x)=\begin{cases}\phi(x), &x\in B_r(x_0)\\u_k(x), &x\not\in B_r(x_0)\end{cases}.$$

By continuity, there exists $x_k\in B_r(x_0) $ such that $u_k(x_k)-\phi(x_k)\ge u_k(x)-\phi(x)$ for all $x\in B_r(x_0)$. Since $\phi$ is strictly touching $u$ at $x_0$ and $u_k$ are converging to $u$ uniformly in $B_r$, we have $x_k\to x_0$. Thus for large $k$ we might assume $|x_k-x_0|\le \frac{1}{2}r$.

Since $u_k$ are solutions, one has $$F_k(D^\sigma\tilde{\phi}_k(x_k))\ge 0.$$

Now for each $(i,j)$, 
\begin{align*}
|D^\sigma\tilde{\phi}_k(x_k)-D^\sigma\tilde{\phi}(x_0)|&=|\int(\delta \tilde{\phi}_k(x_k,y)-\delta \tilde{\phi}(x_0,y))\frac{\langle e_i,y\rangle\langle e_j,y\rangle}{|y|^{n+\sigma+2}}dy|\\&=|\int_{|y|\le\frac{1}{2}r}+\int_{|y|>\frac{1}{2}r}(\delta \tilde{\phi}_k(x_k,y)-\delta \tilde{\phi}(x_0,y))\frac{\langle e_i,y\rangle\langle e_j,y\rangle}{|y|^{n+\sigma+2}}dy|\\&\le|\int_{|y|\le\frac{1}{2}r}(\delta\phi(x_k,y)-\delta\phi(x_0,y))\frac{\langle e_i,y\rangle\langle e_j,y\rangle}{|y|^{n+\sigma+2}}dy|\\&+|\int_{|y|>\frac{1}{2}r}(\delta \tilde{\phi}_k(x_k,y)-\delta \tilde{\phi}(x_0,y))\frac{\langle e_i,y\rangle\langle e_j,y\rangle}{|y|^{n+\sigma+2}}dy|.
\end{align*}

Now for the first term we can use the smoothness of $\phi$, and for the second we can use the integrability of $\frac{1}{|y|^{n+\sigma}}$ to see both are converging to $0$. As a result, uniform convergence of $F_k$ to $F$ gives  
$$F(D^\sigma\tilde{\phi}(x))\ge 0.$$

Therefore, $u$ is a subsolution. By similar argument one shows that $u$ is also a supersolution.
\end{proof}

The following Liouville-type theorem is the key to study blow-up limit of solutions. Similar results were first shown by Serra in \cite{Ser1}\cite{Ser2} to deal with the lack of control for boundary data of nonlocal equations.

\begin{thm}
Let $0<\alpha'<\alpha<1$ be such that $\alpha'+\sigma$ and $\alpha+\sigma$ have the same integer part $\nu$. Also $\alpha'+\sigma-\nu$ and $\alpha+\sigma-\nu$ are both non-zero.

Suppose for each $0\le\beta\le\sigma+\alpha'$ one has the growth condition \begin{equation}
[u]_{C^{\beta}(B_R)}\le R^{\sigma+\alpha-\beta} \text{  for $R\ge 1$},
\end{equation} and \begin{equation}
\int \delta(u(\cdot+h)-u)(x,y)\frac{1}{|y|^{n+\sigma}}dy=0 \text{  for $h\in\mathbb{R}^n$}.
\end{equation} Then $u$ is a polynomial of degree $\nu$.
\end{thm} 

\begin{rem}
Throughout this paper, $[\cdot]_{C^\beta}$ is to be understood as $[\cdot]_{C^{\nu_\beta,\beta-\nu_\beta}}$, where $\nu_\beta$ is the integer part of $\beta$. 
\end{rem} 

\begin{rem}
We will need to study functions of the form $u=\tilde{u}-p$, where $\tilde{u}$ is the solution to some equation and $p$ is a polynomial. In particular, $u$ may not have the correct decay at infinity for the fractional order operator to be well-defined. 

If $p$ is affine, this does not pose any serious problem. Since the symmetric difference does not see affine perturbations, one has $\delta u(x,y)=\delta\tilde{u}(x,y)$ and thus $u$ and $\tilde{u}$ solve the same equation. 

However, if $p$ is a paraboloid, then we need to study $u(\cdot+h)-u(\cdot)$ instead. Here we take advantage of the identity $\delta(p(\cdot+h))(x,y)-\delta p(x,y)=0$.
\end{rem}

\begin{rem}
Similar results in \cite{Ser2} can only handle $\alpha<\bar{\alpha}$, where $\bar{\alpha}$ is some universal constant. We have a stronger result that holds for all $\alpha\in(0,1)$ since our equation for $u(\cdot+h)-u(\cdot)$ is the fractional Laplacian.
\end{rem}
\begin{proof}We leave the case $\nu=0$ to the reader.

\textbf{When $\nu=1$.} 

Since $\sigma+\alpha'>1$, we can take $\beta=1$ in the growth estimate to obtain $$\|u_e\|_{\mathcal{L}^{\infty}(B_R)}\le R^{\sigma+\alpha-1} \text{ for $R\ge 1$ and $e\in\mathbb{S}^{n-1}$}.$$

In particular $\|u_e\|_{\mathcal{L}^{\infty}(B_1)}\le 1$ and $$\int_{B_1^c} |u_e|/|y|^{n+\sigma}dy\le C(n)\int_1^\infty r^{\sigma+\alpha-1}r^{n-1}/r^{n+\sigma}dr\le C(n).$$

By the equation one has $$ \int \delta(\frac{u(\cdot+\epsilon e)-u}{|\epsilon|})(x,y)\frac{1}{|y|^{n+\sigma}}dy=0.$$ Taking $\epsilon\to 0$ one has $$\int \delta u_e(x,y)\frac{1}{|y|^{n+\sigma}}dy=0.$$

Consequently the $C^{1,\alpha}$-estimate \cite{CS1} gives $$\|\nabla u_e\|_{\mathcal{L}^{\infty}(B_{1/2})}\le C(n).$$

For $\rho\ge 1$ define $v(x)=\rho^{-(\sigma+\alpha-1)}u_e(\rho x)$. 

Then $v$ satisfies the same growth estimate as $u_e$: $$\|v\|_{\mathcal{L}^{\infty}(B_R)}=\rho^{-(\sigma+\alpha-1)}\|u_e\|_{\mathcal{L}^{\infty}(B_{\rho R})}\le \rho^{-(\sigma+\alpha-1)}(\rho R)^{\sigma+\alpha-1}=R^{\sigma+\alpha-1}. $$  

Obviously $v$ also solves the same equation. As a result the same $C^{1,\alpha}$-estimate gives $$\|\nabla v\|_{\mathcal{L}^{\infty}(B_{1/2})}\le C(n).$$

Now note that $$\|\nabla v\|_{\mathcal{L}^{\infty}(B_{1/2})}=\rho^{-(\sigma+\alpha-1)}\cdot \rho\|\nabla u_e\|_{\mathcal{L}^{\infty}(B_{1/2\rho})}.$$ One has 
$$\|\nabla u_e\|_{\mathcal{L}^{\infty}(B_{1/2\rho})}\le C(n)\rho^{\sigma+\alpha-2}.$$

Since $\nu=1$, $\sigma+\alpha-2<0$. Thus $\rho\to\infty$ shows that $u_e$ is a constant. This being true for all $e\in\mathbb{S}^{n-1}$, we see that $u$ is affine.

\textbf{When $\nu=2$.}

Now we take instead $\beta=2$ and obtain $$\|D^2u\|_{\mathcal{L}^{\infty}(B_R)}\le R^{\sigma+\alpha-2} \text{ if $R\ge 1$}.$$

Meanwhile the equation gives $$\int \delta(\frac{u(\cdot+\epsilon e)+u(\cdot-\epsilon e)-2u}{|\epsilon|^2})(x,y)\frac{1}{|y|^{n+\sigma}}dy=0.$$ 

When $\epsilon\to 0$ one has $$\int \delta D_{ee}u(x,y)\frac{1}{|y|^{n+\sigma}}dy=0.$$

From here one uses the same argument as before to obtain $D_{ee}u$ is constant for all $e\in\mathbb{S}^{n-1}$ and as a result $u$ is a paraboloid.
\end{proof} 

We need to use some previously known regularity result in \cite{CS1}. Hence it is important that our operator falls into their category of nonlocal elliptic operators:

\begin{prop}
The operator $u\mapsto F(D^\sigma u)$ is elliptic with respect to $\mathcal{L}_1$ as in Caffarelli-Silvestre \cite{CS1}.
\end{prop} 

\begin{proof}
For smooth bounded functions $\phi$ and $\psi$, the ellipticity of $F$ leads to 
\begin{align*}F(D^{\sigma}\phi(x))-F(D^{\sigma}\psi(x))&\le\sup_{\lambda\le A\le\Lambda} \Sigma A_{ij}(D^{\sigma}_{ij}\phi(x)-D^{\sigma}_{ij}\psi(x))\\&=\sup_{\lambda\le A\le\Lambda} \Sigma A_{ij}(\int\delta\phi(x,y)\frac{\langle e_i,y\rangle\langle e_j,y\rangle}{|y|^{n+\sigma+2}}dy-\int\delta\psi(x,y)\frac{\langle e_i,y\rangle\langle e_j,y\rangle}{|y|^{n+\sigma+2}}dy)\\&=\sup_{\lambda\le A\le\Lambda}\int(\delta\phi(x,y)-\delta\psi(x,y))\frac{\Sigma A_{ij}\langle e_i,y\rangle\langle e_j,y\rangle}{|y|^{n+\sigma+2}}dy\\&=\sup_{\lambda\le A\le\Lambda}\int(\delta\phi(x,y)-\delta\psi(x,y))\frac{\langle Ay,y\rangle}{|y|^{n+\sigma+2}}dy.
\end{align*} Now note that the kernel $\frac{\langle Ay,y\rangle}{|y|^{n+\sigma+2}}$ is in the class $\mathcal{L}_1$ as in \cite{CS1}, we have $$F(D^{\sigma}\phi(x))-F(D^{\sigma}\psi(x))\le M^{+}_{\mathcal{L}_1}(\phi-\psi)(x).$$ The symmetric inequality follows from symmetric argument.
\end{proof}

\section{Improvement of regularity}

The following improvement of regularity result is a key stepping stone for the main result. It says that given $\alpha\in(0,1)$, any sub-optimal regularity can be improved to the optimal $C^{\sigma+\alpha}$-regularity if the solution is small enough. 

\begin{thm} Let $F$ be as in Theorem 1.2.

Given $\alpha\in (0,1)$ and $0<\alpha'<\alpha<1$ such that $\sigma+\alpha$ and $\sigma+\alpha'$ are not integers and have the same integer part $\nu$.  Then there exist constants $\kappa=\kappa(n, \lambda, \Lambda,\sigma, \alpha, \alpha',\omega)>0$ and $C=C(n, \lambda, \Lambda, \sigma, \alpha, \alpha',\omega)<\infty$ such that if $$F(D^\sigma u)=0 \text{ in $B_1$}$$ and $$[u]_{C^{\sigma+\alpha'}(B_1)}+\|u\|_{\mathcal{L}^{\infty}(B_1)}+\|u\|_{\mathcal{L}^1(\frac{1}{|y|^{n+\sigma}}dy)}\le\kappa,$$ then we have the following $$[u]_{C^{\sigma+\alpha}(B_{1/2})}\le C ([u]_{C^{\sigma+\alpha'}(B_1)}+\|u\|_{\mathcal{L}^{\infty}(B_1)}+\|u\|_{\mathcal{L}^1(\frac{1}{|y|^{n+\sigma}}dy)}).$$
\end{thm} 

The following simple lemma plays an important role in these improvement of regularity arguments. Compare to the one in \cite{Ser2}, we have truncated the length scale here. This frees us from excessive assumption on the regularity of $u$ outside $B_1$.

\begin{lem}$0<\alpha'<\alpha<1$.
Suppose $u\in C^{\alpha'}(B_1)$ and $$\sup_{0<r<1/2}\sup_{z\in B_{1/2}}r^{\alpha'-\alpha}[u]_{C^{\alpha'}(B_r(z))}\le A,$$ then $$[u]_{C^{\alpha}(B_{1/2})}\le 2A.$$
\end{lem} 

\begin{proof}
For $z\in B_{1/2}$ and $h\in B_{1/2}$, $$\frac{|u(z+h)-u(z)|}{|h|^{\alpha}}\le \frac{[u]_{C^{\alpha'}(B_{|h|}(z))}|h|^{\alpha'}}{|h|^{\alpha}}=|h|^{\alpha'-\alpha}[u]_{C^{\alpha'}(B_{|h|}(z))}\le A.$$  

Now for $z,z'\in B_{1/2}$ but $|z-z'|>1/2$, let $m=\frac{1}{2}(z+z')$. Then the previous estimate gives 
$|u(z)-u(m)|\le A|z-m|^{\alpha}$ and $|u(z')-u(m)|\le A|z'-m|^{\alpha}$. Thus 
$$|u(z)-u(z')|\le A(|z-m|^{\alpha}+|z'-m|^{\alpha})=A2^{1-\alpha}|z-z'|^{\alpha}.$$
\end{proof} 

We now begin the proof of Theorem 3.1. 

Suppose the result is false, then we can find a sequence of operators $F_k$ with ellipticity constant $\lambda$ and $\Lambda$, $F_k(0)=0$, and $DF_k$ enjoy the same modulus of continuity $\omega$.  

There is also a sequence of functions $u_k$ such that $$[u_k]_{C^{\sigma+\alpha'}(B_1)}+\|u_k\|_{\mathcal{L}^{\infty}(B_1)}+\|u_k\|_{\mathcal{L}^1(\frac{1}{|y|^{n+\sigma}}dy)}=\kappa_k\to 0,$$ $$F_k(D^\sigma u_k)=0 \text{ in $B_1$}$$ but $$[u_k]_{C^{\sigma+\alpha}(B_{1/2})}\ge k\kappa_k.$$

With an unforgivable abuse of notation, we would denote $\frac{1}{\kappa_k}u_k$ by the same function $u_k$. Then one has the following \begin{equation}
\frac{1}{\kappa_k}F_k(\kappa_k D^\sigma u_k)=0 \text{ in $B_1$},
\end{equation} 
\begin{equation}
[u_k]_{C^{\sigma+\alpha'}(B_1)}+\|u_k\|_{\mathcal{L}^{\infty}(B_1)}+\|u_k\|_{\mathcal{L}^1(\frac{1}{|y|^{n+\sigma}}dy)}=1,
\end{equation} and \begin{equation}
[u_k]_{C^{\sigma+\alpha}(B_{1/2})}\ge k.
\end{equation} 

Define, for each $k$, the following quantity $$\theta_k(r'):=\sup_{r'<r<1/2}\sup_{z\in B_{1/2}}r^{\alpha'-\alpha}[u_k]_{C^{\sigma+\alpha'}(B_r(z))}.$$ We'd like to point out that a similar quantity was studied in \cite{Ser2} to get the `correct rate' of blow-up.

Then Lemma 3.1 gives $$\lim_{r'\to 0}\theta_k(r')=\sup_{r'>0}\theta_k(r')\ge k/2.$$

By definition, one finds $r_k>1/k$ and $z_k\in B_{1/2}$ such that \begin{equation}
r_k^{\alpha'-\alpha}[u_k]_{C^{\sigma+\alpha'}(B_{r_k}(z_k))}>\frac{1}{2}\theta_k(\frac{1}{k})\ge \frac{1}{2}\theta_k(r_k)\to\infty.
\end{equation} 
Note that $[u_k]_{C^{\sigma+\alpha'}(B_{r_k}(z_k))}\le 1$, the above estimate forces $r_k\to 0$.

Define the blow-up sequence $$v_k(x):=\frac{1}{\theta_k(r_k)}\frac{1}{r_k^{\sigma+\alpha}}u_k(r_kx+z_k).$$

Now we divide the proof into two cases depending on the value of $\nu$.  The key difference is that for $\nu\le 1$, we would subtract from $v_k$ an affine function. Here the new function would solve the same equation since the symmetric difference does not see affine functions. For $\nu=2$ we would subtract instead a paraboloid. This case takes more effort since the function obtained will not solve the same equation. 

\begin{proof} \textbf{The Case $\nu\le 1.$}

Actually we only give the proof for $\nu=1$, leaving the case $\nu=0$ to the reader. 

For $1\le R\le \frac{1}{2r_k}$ one has
\begin{align*}
[v_k]_{C^{\sigma+\alpha'}(B_R)}&=\frac{1}{\theta_k(r_k)}\frac{1}{r^{\sigma+\alpha}_k}[u_k]_{C^{\sigma+\alpha'}(B_{r_kR}(z_k))}r_k^{\sigma+\alpha'}\\&=\frac{1}{\theta_k(r_k)}r_k^{\alpha'-\alpha}[u_k]_{C^{\sigma+\alpha'}(B_{r_kR}(z_k))}\\&=\frac{1}{\theta_k(r_k)}\frac{(r_kR)^{\alpha'-\alpha}}{R^{\alpha'-\alpha}}[u_k]_{C^{\sigma+\alpha'}(B_{r_kR}(z_k))}\\&\le\frac{1}{\theta_k(r_k)}\theta_k(r_kR)R^{\alpha-\alpha'}\\&\le R^{\alpha-\alpha'}.
\end{align*}
Here we used the monotonicity of $\theta$.

Define $\ell_k(x)=v(0)+\nabla v(0)\cdot x$ and $\tilde{v}_k=v_k-\ell_k.$ 

Since $[\ell_k]_{C^{\sigma+\alpha'}}=0$ the previous estimate passes to $\tilde{v}_k$:\begin{equation}
[\tilde{v}_k]_{C^{\sigma+\alpha'}(B_R)}\le R^{\alpha-\alpha'} \text{  for $1\le R\le\frac{1}{2r_k}$}.
\end{equation} 

In particular by picking $R=1$, one has $$[\nabla\tilde{v}_k]_{C^{\sigma+\alpha'-1}}(B_1)\le 1.$$

With $|\nabla\tilde{v}_k(0)|=0$, this implies $|\nabla\tilde{v}_k(x)|\le |x|^{\sigma+\alpha'-1}$ for $x\in B_1$, and as a result,

\begin{align*}|\tilde{v}_k(x)|&=|\tilde{v}_k(x)-\tilde{v}_k(0)|\\&\le \|\nabla\tilde{v}_k\|_{\mathcal{L}^{\infty}(B_{|x|})}|x|\\&\le |x|^{\sigma+\alpha'}
\end{align*}for all $x\in B_1$.

Note that this estimate is independent of $k$, and this is the reason why one needs to subtract the affine function $\ell_k$ from $v_k$.

Now let $\eta$ be some cut-off function that is $1$ in $B_{1/2}$ but vanishes outside $B_1$.  For each $e\in\mathbb{S}^{n-1}$,$$|\int_{B_{1}} D_e\tilde{v}_k\eta|=|\int_{B_{1}} \tilde{v}_kD_e\eta|\le C(n).$$ Thus one can find $\bar{x}\in B_1$ such that $|D_e\tilde{v}_k(\bar{x})|\le C(n)$.

By the growth estimate $[D_e\tilde{v}_k]_{C^{\sigma+\alpha'-1}(B_R)}\le R^{\alpha-\alpha'}$, one has for $x\in B_R$, $1\le R\le \frac{1}{2r_k}$ \begin{align*}
|D_e\tilde{v}_k(x)|&\le|D_e\tilde{v}_k(\bar{x})|+R^{\alpha-\alpha'}|x-\bar{x}|^{\sigma+\alpha'-1}\\&\le C(n)+R^{\alpha-\alpha'}(R+1)^{\sigma+\alpha'-1}\\&\le CR^{\sigma+\alpha-1}.
\end{align*}

Now an interpolation argument gives $$
[\tilde{v}_k]_{C^{\beta}(B_R)}\le CR^{\sigma+\alpha-\beta}$$for $\beta\le\sigma+\alpha'$ and $1\le R\le \frac{1}{2r_k}$.

To see this, note that for $1\le|x|\le R$, \begin{align*}|\tilde{v}_k(x)|&=|\tilde{v}_k(x)-\tilde{v}_k(0)|\\&\le \|\nabla\tilde{v}_k\|_{\mathcal{L}^{\infty}(B_{|x|})}|x|\\&\le C|x|^{\sigma+\alpha-1}|x|\\&\le C|x|^{\sigma+\alpha}\\&\le CR^{\sigma+\alpha}.\end{align*} This, together with the estimate in $B_1$, gives the growth estimate when $\beta=0$. 

For $\beta\in(0,1)$, the estimate follows from 
\begin{align*}|\tilde{v}_k(x')-\tilde{v}_k(x)|&\le |x'-x| \|\nabla\tilde{v}_k\|_{\mathcal{L}^{\infty}(B_R)}\\&\le |x'-x|^{\beta}|x'-x|^{1-\beta} \|\nabla\tilde{v}_k\|_{\mathcal{L}^{\infty}(B_R)}\\&\le |x'-x|^{\beta}R^{1-\beta} CR^{\sigma+\alpha-1}\\&=CR^{\sigma+\alpha-\beta}|x'-x|^{\beta}.\end{align*}

We have already established the case $\beta=1$, and the case $\beta\in (1,\sigma+\alpha']$ follows from similar argument.

This growth estimate implies that,  up to a subsequence, $\tilde{v}_k$ converges locally uniformly in $C^{\beta}$ to some $\tilde v$ for any $\beta<\sigma+\alpha'$. The previous growth estimate passes to $\tilde{v}$ for any $R\ge 1$:\begin{equation}
[\tilde{v}]_{C^{\beta}(B_R)}\le CR^{\sigma+\alpha-\beta}.
\end{equation} 

Meanwhile, \begin{align*}[\tilde{v}_k]_{C^{\sigma+\alpha'}(B_1)}&=\frac{1}{\theta_k(r_k)}\frac{1}{r_k^{\sigma+\alpha}}r_k^{\sigma+\alpha'}[u_k]_{C^{\sigma+\alpha'}(B_{r_k}(z_k))}\\&=\frac{1}{\theta_k(r_k)}r_k^{\alpha'-\alpha}[u_k]_{C^{\sigma+\alpha'}(B_{r_k}(z_k))}\\&\ge 1/2.
\end{align*} This also passes to $\tilde{v}$: \begin{equation}
[\tilde{v}]_{C^{\sigma+\alpha'}(B_1)}\ge 1/2.
\end{equation} 

Now note that $$D^\sigma v_k(x)=\frac{1}{\theta_k(r_k)}\frac{1}{r_k^{\alpha}}D^\sigma u_k(r_kx+z_k),$$ thus $v_k$ solves $$\frac{1}{\theta_k(r_k)r_k^{\alpha}\kappa_k}F_k(\theta_k(r_k)r_k^{\alpha}\kappa_kD^\sigma v_k(x))=0 \text{ in $B_{\frac{1}{2r_k}}$}.$$

Since \begin{align*}
\theta_k(r_k)r_k^{\alpha}&=r_k^\alpha\sup_{r_k<r<1/2}\sup_{z\in B_{1/2}}r^{\alpha'-\alpha}[u_k]_{C^{\sigma+\alpha'}(B_{r}(z))}\\&\le r_k^{\alpha}r_k^{\alpha'-\alpha}[u_k]_{C^{\sigma+\alpha'}(B_{1})}\\&\le 1,
\end{align*} we have $$\theta_k(r_k)r_k^{\alpha}\kappa_k\to 0.$$

Thus Proposition 2.2 and 2.3 apply and give an $A\in S^n$ such that $\Sigma A_{ij}D^{\sigma}_{ij}\tilde{v}=0$ in $\mathbb{R}^n$. Up to an affine change of variables,

$$\int \delta\tilde{v}(x,y)\frac{1}{|y|^{n+\sigma}}dy=0 \text{ in $\mathbb{R}^n$}.$$

Combining this with (3.6), we apply Theorem 2.4 to conclude $\tilde{v}$ is affine, contradicting (3.7).
 \end{proof} 
 
 Now we proceed to the case when $\nu=2$.

\begin{proof}\textbf{The Case $\nu=2$.}

With similar arguments one establishes for $1\le R\le \frac{1}{2r_k}$ 
\begin{equation*}
[v_k]_{C^{\sigma+\alpha'}(B_R)}\le R^{\alpha-\alpha'}.
\end{equation*}
Now define $p_k(x)=v_k(0)+\nabla v_k(0)+\frac{1}{2}\langle D^2v(0)x,x\rangle$ and $\tilde{v}_k=v_k-p_k$. 

Again one has the same growth estimate on $\tilde{v}_k$:
 $$
[\tilde{v}_k]_{C^{\beta}(B_R)}\le CR^{\sigma+\alpha-\beta}$$for $\beta\le\sigma+\alpha'$ and $1\le R\le \frac{1}{2r_k}$. 

And up to a subsequence $\tilde{v}_k$ converges locally uniformly to $\tilde{v}$ in $C^{\beta}$ for any $\beta<\sigma+\alpha'$. The limit satisfies \begin{equation}
[\tilde{v}]_{C^{\beta}(B_R)}\le CR^{\sigma+\alpha-\beta}
\end{equation} for $0\le\beta\le\sigma+\alpha'$ and $R\ge 1$.

As before one has \begin{equation}
[\tilde{v}]_{C^{\sigma+\alpha'}(B_1)}\ge 1/2.
\end{equation} 

Now comes the major difference from the previous case, namely, we do not have $\delta\tilde{v}_k(x,y)=\delta v_k(x,y)$ anymore. 

However, note that for a paraboloid $p(x)=p(0)+\langle b,x\rangle+\frac{1}{2}\langle Ax,x\rangle$, one has \begin{align*}\delta p(x,y)&=(p(0)+\langle b,x+y\rangle+\frac{1}{2}\langle A(x+y),x+y\rangle)+(p(0)+\langle b,x-y\rangle+\frac{1}{2}\langle A(x-y),x-y\rangle)\\&-2(p(0)+\langle b,x\rangle+\frac{1}{2}\langle Ax,x\rangle)\\&=\frac{1}{2}\langle A(x+y),x+y\rangle+\frac{1}{2}\langle A(x-y),x-y\rangle-\langle Ax,x\rangle\\&=\frac{1}{2}\langle Ax,x\rangle+\frac{1}{2}\langle Ay,y\rangle+\langle Ax,y\rangle +\frac{1}{2}\langle Ax,x\rangle+\frac{1}{2}\langle Ay,y\rangle-\langle Ax,y\rangle-\langle Ax,x\rangle\\&=\langle Ay,y\rangle.
\end{align*}

Therefore, if we define $$p^h(x):=p(x+h)=p(0)+\langle b,h\rangle+\frac{1}{2}\langle Ah,h\rangle+\langle b,x\rangle+\langle Ah,x\rangle+\frac{1}{2}\langle Ax,x\rangle,$$ then $$\delta p^h(x,y)=\langle Ay,y\rangle=\delta p(x,y).$$

Apply this to our functions, we have \begin{align*}\delta\tilde{v}^h_k(x,y)-\delta\tilde{v}_k(x,y)&=\delta v^h_k(x,y)-\delta v_k(x,y)-(\delta p^h_k(x,y)-\delta p_k(x,y))\\&=\delta v^h_k(x,y)-\delta v_k(x,y).\end{align*}In particular, although $\tilde{v}_k$ and $\tilde{v}^h_k$ grow too fast at infinity for the $\sigma$-order Hessian to be defined, due to the above cancelation, $D^\sigma (\tilde{v}^h_k-\tilde{v}_k)(x)$ is well-defined and $$D^\sigma (\tilde{v}^h_k-\tilde{v}_k)(x)=D^\sigma v^h_k(x)-D^\sigma v_k(x).$$

We now study the equation satisfied by $\tilde{v}^h_k-\tilde{v}_k$.

Since $$\frac{1}{\theta_k(r_k)r_k^{\alpha}\kappa_k}F_k(\theta_k(r_k)r_k^{\alpha}\kappa_kD^\sigma v^h_k(x))=0,$$ we have $$\frac{1}{\theta_k(r_k)r_k^{\alpha}\kappa_k}F_k(\theta_k(r_k)r_k^{\alpha}\kappa_k(D^\sigma (\tilde{v}^h_k-\tilde{v}_k)(x)+D^\sigma v_k(x)))=0.$$

Let $\epsilon_k:=\theta_k(r_k)r_k^{\alpha}\kappa_k$, then as in the previous case, $$\epsilon_k=O(\kappa_k),$$ and \begin{equation}\frac{1}{\epsilon_k}F_k(\epsilon_k(D^\sigma (\tilde{v}^h_k-\tilde{v}_k)(x)+D^\sigma v_k(x)))=0.\end{equation}

We claim \begin{lem}
Up to an affine change of variables, $$\int \delta (\tilde{v}^h-\tilde{v})(x,y)\frac{1}{|y|^{n+\sigma}}dy=0.$$
\end{lem} 
\begin{proof}
Again let $A\in S^n$ be the limit of $DF_k(0)$.

(3.10) implies 
\begin{align*}0=&\frac{1}{\epsilon_k}F_k(\epsilon_k(D^\sigma (\tilde{v}^h_k-\tilde{v}_k)(x)+D^\sigma v_k(x)))-\frac{1}{\epsilon_k}F_k(\epsilon_k D^\sigma v_k(x))\\=&\frac{1}{\epsilon_k}\int_0^1\frac{d}{dt}F_k(t\epsilon_k(D^\sigma (\tilde{v}^h_k-\tilde{v}_k)(x)+D^\sigma v_k(x))+(1-t)\epsilon_k D^\sigma v_k(x))dt\\=&\frac{1}{\epsilon_k}\int_0^1DF_k(t\epsilon_k(D^\sigma (\tilde{v}^h_k-\tilde{v}_k)(x)+D^\sigma v_k(x))+(1-t)\epsilon_k D^\sigma v_k(x))dt\cdot\epsilon_kD^\sigma (\tilde{v}^h_k-\tilde{v}_k)(x)\\=&\int_0^1DF_k(\epsilon_k(tD^\sigma (\tilde{v}^h_k-\tilde{v}_k)(x)+D^\sigma v_k(x)))dt\cdot D^\sigma (\tilde{v}^h_k-\tilde{v}_k)(x).
\end{align*}

Since $\tilde{v}^h_k-\tilde{v}_k\to \tilde{v}^h-\tilde{v}$ locally uniformly, by Proposition 2.3, it suffices to show that $$\int_0^1DF_k(\epsilon_k(tD^\sigma (\tilde{v}^h_k-\tilde{v}_k)(x)+D^\sigma v_k(x)))dt\to A \text{ uniformly.}$$ 
To see this, note that \begin{align*}|D^\sigma v_k(x)|=&\frac{1}{\theta_k(r_k)r_k^{\sigma+\alpha}}r_k^{\sigma}|D^\sigma u_k(r_kx+z_k)|\\=&\frac{1}{\theta_k(r_k)r_k^{\alpha}}|D^\sigma u_k(r_kx+z_k)|\\\le&\frac{1}{\theta_k(r_k)r_k^{\alpha}}.\end{align*} We used $$[u_k]_{C^{\sigma+\alpha'}(B_1)}+\|u_k\|_{\mathcal{L}^{\infty}(\mathbb{R}^n)}=1$$ for the last inequality.

Similar estimate holds for $D^\sigma v^h_k(x)$ and hence for $D^\sigma (\tilde{v}^h_k-\tilde{v}_k)(x)$.

Consequently, 

\begin{align*}
|\int_0^1DF_k(\epsilon_k(tD^\sigma (\tilde{v}^h_k-\tilde{v}_k)(x)+D^\sigma v_k(x)))dt-A|\le&|DF_k(0)-A|\\+|\int_0^1DF_k(\epsilon_k(tD^\sigma (\tilde{v}^h_k-\tilde{v}_k)(x)+&D^\sigma v_k(x)))-DF_k(0)dt|\\\le& |DF_k(0)-A|+\omega(\kappa_k).
\end{align*}

This completes the proof for the lemma.
\end{proof} 

This lemma combined with (3.8) and Theorem 2.4 shows that $\tilde{v}$ is a paraboloid, contradicting (3.9). 

This completes our proof for Theorem 3.1.

\end{proof}

\section{Proof of Theorem 1.2}
We now combine previously known regularity estimates and the improvement of regularity to complete the proof of the main result. 

The starting point is:

\begin{prop}
There is some universal $\bar{\alpha}>0$ and $\bar{C}<\infty$ such that viscosity solution to $$F(D^\sigma u)=0 \text{ in $B_1$}$$ satisfies $$\|u\|_{C^{1,\bar{\alpha}}(B_{1/2})}\le \bar{C}(\|u\|_{\mathcal{L}^\infty(B_1)}+\|u\|_{\mathcal{L}^1(\frac{1}{1+|y|^{n+\sigma}}dy)}).$$
\end{prop} 

\begin{proof}
This is a direct consequence of Proposition 2.8 and estimates in \cite{CS1}.
\end{proof} 

This leads to the following, which is the main result for small exponents:

\begin{prop}
Let $F$ be as in Theorem 1.2.  

For any $\alpha\in (0,1)$ with $\sigma+\alpha<2$ and not an integer, there are constants $\kappa'>0$ and $C<\infty$ such that if $u$ solves $$F(D^\sigma u)=0 \textit{ in $B_1$}$$ and $\|u\|_{\mathcal{L}^\infty(B_1)}+\|u\|_{\mathcal{L}^1(\frac{1}{1+|y|^{n+\sigma}}dy)}\le\kappa'$, then $$[u]_{C^{\sigma+\alpha}(B_{1/2})}\le C(\|u\|_{\mathcal{L}^\infty(B_1)}+\|u\|_{\mathcal{L}^1(\frac{1}{1+|y|^{n+\sigma}}dy)}).$$
\end{prop}

\begin{proof}
The case $\sigma+\alpha\le 1+\bar{\alpha}$ is covered by the previous estimate, hence we only deal with the case where $1<\sigma+\alpha<2.$ 

In this case we fix $\alpha'<\alpha$ such that $1\le\sigma+\alpha'\le 1+\bar{\alpha}$. Then obviously $\sigma+\alpha'$ and $\sigma+\alpha$ have the same integer part, $1$.

By Proposition 4.1, $[u]_{C^{\sigma+\alpha'}(B_{3/4})}\le \bar{C}\kappa'$. Thus by taking $\kappa'$ small enough depending only on $\kappa$ as in Theorem 3.1 and $\bar{C}$ as in Proposition 4.1, we have 

$$[u]_{C^{\sigma+\alpha'}(B_{3/4})}+\|u\|_{\mathcal{L}^\infty(B_1)}+\|u\|_{\mathcal{L}^1(\frac{1}{1+|y|^{n+\sigma}}dy)}\le\kappa$$ and hence Theorem 3.1 applies to $u$ and gives the desired estimate.
\end{proof} 

\begin{rem}
Comparing with Proposition 4.1, the key point is that instead of some small universal $\bar{\alpha}$, we now essentially have $C^{1,\alpha}$-estimate for all $\alpha\in(0,1)$ for small solutions. In particular it says all small solutions are classical solutions since they are in $C^{\sigma^+}$.
\end{rem} 

We use the previous result to complete the proof of the main result:

\begin{proof}
We are left to deal with the case when $\sigma+\alpha>2$. 

Note that one has $2>\sigma>1$ in this case. 

By the previous result, we find $\kappa''>0$ and $\tilde{\alpha}>0$ such that $1<\sigma+\tilde{\alpha}<2$ and $u\in C^{\sigma+\tilde{\alpha}}$ whenever $\|u\|_{\mathcal{L}^\infty(B_1)}+\|u\|_{\mathcal{L}^1(\frac{1}{1+|y|^{n+\sigma}}dy)}\le\kappa''.$

Consequently $u_e$ is well-defined and satisfies $$\Sigma F_{ij}(D^\sigma u(x))D^{\sigma}_{ij}u_e(x)=0 \text{ in $B_1$}.$$ In a more familiar form, this is $$\int \delta u_e(x,y)\frac{\langle F_{ij}(D^\sigma u(x))y,y\rangle}{|y|^{n+\sigma+2}}dy=0\text{  in $B_1$}.$$

Now note that $F_{ij}(D^{\sigma}u(\cdot))$ is in $C^{\tilde{\alpha}}$ \cite{Yu2}, we can use cut-off argument in nonlocal Schauder theory \cite{JX} \cite{Ser2} to have estimate on $[u]_{C^{1+\sigma+\tilde{\alpha}}(B_{1/4})}=[u_e]_{C^{\sigma+\tilde{\alpha}}(B_{1/4})}.$ Note that $1+\sigma+\tilde{\alpha}$ and $\sigma+\alpha$ now have the same integer part 2,  we can again use the improvement of regularity result to obtain the desired result.

\end{proof} 

\section*{Acknowledgement}
The author is grateful to his PhD advisor Luis Caffarelli, for his constant encouragement and guidance, and Dennis Kriventsov for many insightful discussions and suggestions.


\end{document}